\documentclass[12pt]{amsart}

\usepackage{geometry}
\usepackage[all]{xy}
\usepackage{epstopdf}
\usepackage{amssymb,amsmath,amsfonts,amsthm,graphicx,enumerate,amscd}
\usepackage{mathrsfs}
\usepackage{slashed}
\usepackage{xcolor}
\usepackage{hyperref}

\newcommand{\deff}[1]{\textbf{\emph{\sharp1}}}

\newcommand{\func}[3]{\sharp1 \colon \sharp2 \to \sharp3}
\newcommand{\bb}[1]{\mathbb{\sharp1}}
\newcommand{\lie}[1]{\mathfrak{\sharp1}}
\newcommand{\iprod}[2]{\langle \sharp1, \sharp2 \rangle}
\newcommand{\ddell}[1]{\frac{\partial}{\partial \sharp1}}

\theoremstyle{plain}
\newtheorem{theorem}{Theorem}[section]

\newtheorem{corollary}[theorem]{Corollary}
\newtheorem{lemma}[theorem]{Lemma}
\newtheorem{proposition}[theorem]{Proposition}

\theoremstyle{definition}
\newtheorem{example}[theorem]{Example}
\newtheorem{remark}[theorem]{Remark}

\title[Adams operations on twisted K-theory]{Adams operations on the twisted K-theory of compact Lie groups}
\author{Chi-Kwong Fok}
\date{February 9, 2024}

\allowdisplaybreaks
\begin{document}
\maketitle
\begin{abstract}
In this paper, extending the results in \cite{F}, we compute Adams operations on the twisted $K$-theory of connected, simply-connected and simple compact Lie groups $G$, in both equivariant and nonequivariant settings.  
\end{abstract}
\tableofcontents
\section{Introduction}
Adams operations are important cohomological operations on $K$-theory. It was utilized by Adams in solving the famous problem of determining the parallelizability of spheres. According to a theorem of Bousfield (cf. \cite[Theorem 9.2]{Bou}), the $\nu_1$-homotopy group, which is a certain localization of the actual homotopy group, can be computed using Adams operations. Motivated by this result, the author of this paper gave explicit formulae for Adams operations on the $K$-theory of all compact classical Lie groups (cf. \cite{F}) in the hope of furthering research on $\nu_1$-homotopy groups of these Lie groups. On the other hand, Adams operations can also be defined similarly on twisted $K$-theory (cf. \cite{AS2}). Just like how Adams operations on the $K$-theory of classical groups can be used to deduce a universal description of Adams operations on $K$-theory in general in terms of homotopy classes of maps on the infinite unitary group, Adams operations on the classical groups in the twisted setting is expected to yield an analogous universal interpretation, which may illuminate the study of Adams operations on twisted $K$-theory. Thus the task of computing Adams operations on twisted $K$-theory of compact Lie groups then becomes natural. By the Freed-Hopkins-Teleman Theorem (cf. \cite{FHT3}), the equivariant twisted $K$-theory of a compact Lie group is isomorphic to its Verlinde algebra, which is distinctively different from its untwisted counterpart (the ring of Grothendieck differentials of the representation ring, cf. \cite{BZ}). Then the Verlinde algebras will provide a new example of algebraic objects equipped with Adams operations. Furthermore, the Verlinde algebra is the representation group of positive energy representations of the loop group of the compact Lie group, and it would be interesting to study how Adams operations act on these infinite dimensional representations in an appropriate sense. 
As a first step towards understanding those problems, we present computations of Adams operations on twisted $K$-theory of compact Lie groups in this paper. 

	Let $G$ be a simple, connected and simply-connected compact Lie group equipped with conjugation action by itself. Any equivariant $K$-theoretic twist on $G$ is classified by $H_G^3(G, \mathbb{Z})\cong\mathbb{Z}$. Let $\tau^n$ be the twist corresponding to $n$, and $\textsf{h}^\vee$ the dual Coxeter number of $G$. According to the Freed-Hopkins-Teleman Theorem (cf. \cite[Theorem 1]{FHT3}), for $n> 0$, the equivariant twisted $K$-theory $K_G^*(G, \tau^n)$ is isomorphic to the Verlinde algebra $R_{n-\textsf{h}^\vee}(G)$, which is the fusion ring for the level $n-\textsf{h}^\vee$ positive energy representation of the loop group $LG$. In Section \ref{equivtwist}, {we review equivariant twisted $K$-theory and the Freed-Hopkins-Teleman Theorem. We also work out in detail the restriction map $i^*: K_G^*(G, \tau^n)\to K_T^*(T, i^*\tau^n)$ to the equivariant twisted $K$-theory of the maximal torus $T$. }
In Section \ref{adamsequivcase}, based on the `splitting principle' for twisted $K$-theory (Corollary \ref{antisym}) provided by the restriction map $i^*$ and an analysis of the equivariant twisted $K$-theory of tori, we obtain Adams operations on equivariant twisted $K$-theory of $G$ below.
\begin{theorem}\label{adamsequiv}
	Let $\ell$ be an integer, $W_\mu\in K_G^*(G, \tau^n)$ the $K$-theory class which is the image of the class of the irreducible $G$-representation $V_\mu\in K_G^*(\text{pt})$ with highest weight $\mu$ under the pushforward map $f_{G, *}: K_G^*(\text{pt})\to K_G^*(G, \tau^n)$ induced by the inclusion of the identity element into $G$, $w_0$ the longest element in the Weyl group of $G$ and $\mu^*:=-w_0\mu$. Denote the half sum of positive roots of $G$ by $\rho$. The Adams operation
	\[\psi^\ell: K_G^*(G, \tau^{n})\to K_G^*(G, \tau^{\ell n})\]
	satisfies 
	\[\psi^\ell(W_\mu)=\begin{cases}W_{\ell\mu+(\ell-1)\rho},&\ \text{if }\ell>0\\ (-1)^{\text{sgn}(w_0)}W_{-\ell\mu^*+(-\ell-1)\rho},&\ \text{if }\ell<0\\ 0,&\ \text{if }\ell=0\end{cases}\]
	for $\mu$ being a level $|n|-\textsf{h}^\vee$ weight and $|n|\geq \textsf{h}^\vee$. It is the zero map if $|n|< \textsf{h}^\vee$ and $n\neq 0$. 
\end{theorem}
The Adams operation $\psi^{-1}$ is in general an involutive isomorphism induced by complex conjugation. Theorem \ref{adamsequiv} in particular gives the explicit isomorphism $\psi^{-1}: K_G^*(G, \tau^n)\to K_G^*(G, \tau^{-n})$ and from this we can show that for $n<0$ $K_G^*(G, \tau^{n})$ as a ring is isomorphic to the Verlinde algebra $R_{|n|-\textsf{h}^\vee}(G)$ (Corollary \ref{zerok}). It is also noteworthy that the above formulae for Adams operations on equivariant twisted $K$-theory appear to be simple in comparison with the complicated combinatorial formulae for the untwisted case (cf. \cite{F}). Moreover, the Adams operation $\psi^0$ is the augmentation map on untwisted $K$-theory and representation rings, whereas it becomes the zero map in the twisted case. 

Adams operations on the nonequivariant twisted $K$-theory turns out to be more subtle. We first review in Section \ref{nonequivtwist} the derivation in \cite{Bra} of $K^*(G, \tau^n)$, which when $|n|\geq\textsf{h}^\vee$ is isomorphic to the exterior algebra on $\text{rank }G-1$ primitive elements over the coefficient group $\mathbb{Z}/c(G, n)$, where $c(G, n)$ is a certain integer invariant of $G$ which can be found in \cite[Equation 3.20 and Table 1]{Bra} or \cite[Theorem 1.2]{D}. We show in Section \ref{nonequivadams} the following formula for Adams operations on the coefficient group.
\begin{theorem}\label{adamsnonequiv}
	Let $k$ be an element in the coefficient group $\mathbb{Z}/c(G, n)$ of $K^*(G, \tau^n)$ and $R_+$ the set of positive roots of $G$. Then 
	\[\psi^\ell(k)=\begin{cases}\ell^{|R_+|}k,&\ \text{if }\ell\geq0\\ (-1)^{\text{sgn}(w_0)+|R_+|}\ell^{|R_+|}k,&\ \text{if }\ell<0\end{cases}\ (\text{mod }c(G, \ell n))\]
	in the coefficient group.
\end{theorem}
We also give an algorithm to compute Adams operations on (wedge product of) primitive generators. It basically involves `resolving' the equivariant Adams operation as in Theorem \ref{adamsequiv} along the Koszul resolution of Verlinde algebras and applying the augmentation map. We illustrate this algorithm with the example $G=SU(3)$ which, though a low-rank example, hopefully gives the reader a sense of subtlety in the description of Adams operations stemming from number theory.

Unless otherwise specified, in the rest of this paper, $G$ is assumed to be a simple, connected and simply-connected compact Lie group.

\textbf{Acknowledgement.} We would like to thank the anonymous referees for the critical comments and especially the suggestions for
improving the exposition and simplifying the proof of Corollary \ref{zerok}. We acknowledge support from the School of Mathematics and Physics Research Grant SRG2324-06 of the Xi'an Jiaotong-Liverpool University.

\section{Equivariant twisted $K$-theory}\label{equivtwist}
\subsection{Review of equivariant twisted $K$-theory}
 In what follows we give a brief review of equivariant twisted $K$-theory. The reader is referred to \cite{AS, M} for more details. Let $H$ be a compact Lie group and $X$ a finite $H$-CW complex. A model for an equivariant twist over $X$, i.e. the local coefficients of the equivariant $K$-theory of $X$, is an $H$-equivariant fiber bundle $P$ whose fibers are projective Hilbert space $\mathbb{P}(\mathcal{H})$, and which is stable in the sense that $P$ is isomorphic to the projective Hilbert space bundle obtained from $P$ by replacing $\mathcal{H}$ with $\mathcal{H}\otimes L^2(H)$ fiberwise. One can associate $P$ with the fiber bundle $\text{Fred}(P)$ of space of Fredholm operators on $\mathcal{H}$ as the structure group $PU(\mathcal{H})$ of $P$ acts on $\text{Fred}(\mathcal{H})$ naturally by conjugation. The equivariant twisted $K$-theory $K_H^0(X, P)$ is defined to be 
 \[[X, \text{Fred}(P)]_H,\]
 the space of homotopy classes of equivariant sections of $\text{Fred}(P)$. When $P$ is trivial, $[X, \text{Fred}(P)]_H$ is the space of homotopy classes of equivariant maps from $X$ to $\text{Fred}(\mathcal{H})$, which gives the untwisted equivariant $K$-theory $K_H^0(X)$. The degree 1 piece $K^1_G(X, P)$ then is defined by suspending $X$ and $P$.
 
 We shall note that another model for twists is equivariant Dixmier-Douady bundles, which are equivariant $\mathcal{K}(\mathcal{H})$-fiber bundles, where $\mathcal{K}(\mathcal{H})$ is the space of compact operators on $\mathcal{H}$. Dixmier-Douady bundles and projective Hilbert space bundles are similar in the sense that both have $PU(\mathcal{H})$ as their structure group. If Dixmier-Douady bundles are used as a model for twists, then equivariant twisted $K$-theory is defined as the $K$-theory of the $C^*$-algebra of the space of sections of the Dixmier-Douady bundle. In what follows, we will use Dixmier-Douady bundles to interpret twists because the orientation twists, which will be needed later, can be more naturally described by Clifford bundles, an example of Dixmier-Douady bundles. This switch to Dixmier-Douady bundles in our exposition can be done without mentioning the $K$-theory of $C^*$-algebras.
 
 Two equivariant Dixmier-Douady bundles $\mathcal{E}_1$ and $\mathcal{E}_2$ are said to be Morita isomorphic if there exists an equivariant Hilbert space bundle $\mathcal{F}\to X$ such that $\mathcal{E}_1\otimes \mathcal{E}_2^*\cong\mathcal{K}(\mathcal{F})$. In this case the Morita isomorphism is said to be trivialized by $\mathcal{F}$. The equivariant twists are classified, up to Morita isomorphism, by $H_H^3(X, \mathbb{Z})$. The tensor product of these bundles corresponds to addition in the cohomology group $H_H^3(X, \mathbb{Z})$. Trivializations of a Morita isomorphism by $\mathcal{F}_1$ and $\mathcal{F}_2$ are said to be equivalent if $\mathcal{F}_1\cong\mathcal{F}_2$. The set of equivalence classes of trivializations of a Morita isomorphism are acted upon transitively by the equivariant Picard group $\text{Pic}_H(X)$ by $[L]\cdot[\mathcal{F}]=[L\otimes\mathcal{F}]$. 
 
Let $X_1$ and $X_2$ be orientable closed differentiable $H$-manifolds and $f: X_1\to X_2$ an equivariant map. The orientation twist $o_i$ over $X_i$ can be represented by the complex Clifford bundle $\mathbb{C}\text{l}(TX_i)$. Its Morita isomorphism class is $\beta w_2^H(X_i)$, where $\beta: H_H^2(X_i, \mathbb{Z}_2)\to H_H^3(X_i, \mathbb{Z})$ is the Bockstein homomorphism and $w_2^H$ is the second equivariant Stiefel-Whitney class. If $\mathcal{E}_1\otimes o_1$ is Morita isomorphic to $f^*(\mathcal{E}_2\otimes o_2)$ and the Morita isomorphism is trivialized by $\mathcal{F}$, then one can define a pushforward map 
\[(f, \mathcal{F})_*: K_H^*(X_1, \mathcal{E}_1)\to K_H^{*+\text{dim}(X_2)-\text{dim}(X_1)}(X_2, \mathcal{E}_2)\]
satisfying $(f, \mathcal{F}\otimes L)_*(a)=(f, \mathcal{F})_*(a\cdot[L])$ for an equivariant line bundle $L\to X_1$. If in addition $H_H^2(X_1, \mathbb{Z})=0$, then there is only one equivalence class of trivializations of the Morita isomorphism and so the pushforward $f_*$ is well-defined. If $\mathcal{E}_1$ and $\mathcal{E}_2$ are Morita isomorphic Dixmier-Douady bundles over $X$ and the isomorphism is trivialized by $\mathcal{F}$, then 
\[(\text{Id}_X, \mathcal{F})_*: K_H^*(X, \mathcal{E}_1)\to K_H^*(X, \mathcal{E}_2)\]
is an isomorphism. Thus the isomorphism class of a twisted $K$-theory group only depends on the Morita isomorphism class of the twist. From now on, we will just denote the twist in any twisted $K$-theory group by its Morita isomorphism class $\tau$ regarded as an equivariant third cohomology class, instead of an explicit Dixmier-Douady bundle.

Let $\alpha\in H^2_H(X_1, \mathbb{Z})$, $\text{dim}(X_1)+\text{dim}(X_2)=2n$ and $f: X_1\to X_2$ be any equivariant differentiable $\alpha$-map, i.e. $\alpha\equiv w^H_2(X_1)-f^*w^H_2(X_2) (\text{mod }2)$. Then the Morita isomorphism class of the Clifford bundle $\mathbb{C}\text{l}(TX_1\oplus f^*TX_2)$ is $0\in H_H^3(X_1, \mathbb{Z})$. This is equivalent to saying that the vector bundle $TX_1\oplus f^*TX_2$ is equivariantly $\text{Spin}^c$. Let $P_{SO(2n)}$ be the orthonormal frame bundle of $TX_1\oplus f^*TX_2$ which is an equivariant principal $SO(2n)$-bundle, and $C_\alpha$ an equivariant circle bundle over $X_1$ whose Chern class is $\alpha$. Form the external tensor product $P_{SO(2n)}\widehat{\otimes}C_\alpha$ which is a principal $SO(2n)\times U(1)$-bundle. Then the equivariant $\text{Spin}^c$ structure of $TX_1\oplus f^*TX_2$ corresponding to $\alpha$ is given by an equivariant principal $\text{Spin}^c(2n)$-bundle $P_{\text{Spin}^c(2n)}^\alpha$ which double covers $P_{SO(2n)}\widehat{\otimes}C_\alpha$. The associated spinor bundle
\[\mathcal{S}_\alpha:=P^\alpha_{\text{Spin}^c(2n)}\times_{\text{Spin}^c(2n)}\textbf{S}, \]
 where $\textbf{S}$ is the spin representation of $\text{Spin}^c(2n)$, Morita trivializes $\mathbb{C}\text{l}(TX_1\oplus f^*TX_2)$ as $\text{End}(\mathcal{S}_\alpha)\cong\mathbb{C}\text{l}(TX_1\oplus f^*TX_2)\cong o_1\otimes f^*o_2$, and gives rise to the $K$-theoretic pushforward $(f, \mathcal{S}_\alpha)_*: K^*_H(X_1, f^*\tau)\to K^*_H(X_2, \tau)$. For convenience we will denote $(f, \mathcal{S}_\alpha)_*$ by $f_*^\alpha$. Let $L_{\alpha'}$ be an equivariant complex line bundle over $X_1$ whose first Chern class is $\alpha'$. Then $\mathcal{S}_\alpha\otimes L_{\alpha'}$, another spinor bundle obtained by scaling the $\text{Spin}^c(2n)$-valued transition functions of $P^\alpha_{\text{Spin}^c(2n)}$ by those of $L_{\alpha'}$, is $\mathcal{S}_{\alpha+2\alpha'}$. This can be seen from the the 2-to-1 homomorphism between the structure groups
 \begin{align*}
 	\text{Spin}^c(2n):=\text{Spin}(2n)\times U(1)/\mathbb{Z}_2&\to SO(2n)\times U(1)\\
	[(g, z)]&\mapsto (\pi(g), z^2),
 \end{align*}
 where $\pi$ is the 2-to-1 homomorphism from $\text{Spin}(2n)$ to $SO(2n)$. It follows that 
	\[f_*^{\alpha+2\alpha'}(a)=f_*^\alpha(a\cdot[L_{\alpha'}])\]
(cf. \cite{AH}, \cite[Theorem 1.1]{CW} for the nonequivariant version). 

Note that $f: X_1\to X_2$ is an equivariant $\alpha$ map if and only if
	\begin{align*}
		\Sigma f: X_1&\to X_2\times\mathbb{R}\\
		x&\mapsto (f(x), 0)
	\end{align*}
	is also an equivariant $\alpha$-map. Let $p: X_2\times\mathbb{R}\to X_2$. When $\text{dim}X_1+\text{dim}X_2$ is odd, then we can define $f_*^\alpha: K_H^*(X_1, f^*\tau)\to K_H^{*-1}(X_2, \tau)$ to be the composition of the `suspended' pushforward $\Sigma f_*^\alpha: K_H^*(X_1, f^*\tau)\to K_H^{*}(X_2\times\mathbb{R}, p^*\tau)$ and the suspension isomorphism $K_H^{*}(X_2\times\mathbb{R}, p^*\tau)\to K_H^{*-1}(X_2, \tau)$. The discussion in the previous paragraph carries over to the case that $\text{dim}X_1+\text{dim}X_2$ is odd, through the map $\Sigma f_*^\alpha$ which falls under the previous case discussed.

Let $i: X_1\to X_2$ be an embedding, the normal bundle $N_{X_2/X_1}$ an equivariant complex vector bundle, and $c=c_1(N_{X_2/X_1})$. Then we have that $i^*(i^c_*(1))$ is the Euler class $e(N_{X_2/X_1})=\sum_{i\geq 0}(-1)^i\left[\bigwedge\nolimits^i_{\mathbb{C}}N_{X_2/X_1}\right]\in K_{H}^0(X_1)$.
 \subsection{The Freed-Hopkins-Teleman Theorem}
 Now assume that $G$ is a simple, connected and simply-connected compact Lie group acting on itself by conjugation. Note that $H_G^3(G, \mathbb{Z})\cong\mathbb{Z}$ (cf. \cite[\S 3.1]{M}). We shall denote the equivariant twist of $G$ corresponding to $n\in H_G^3(G, \mathbb{Z})$ by the multiplicative notation $\tau^n$. The equivariant twisted $K$-theory $K_G^{*}(G, \tau^n)$ is a ring whose ring structure is given by the Pontryagin product, which is induced by the group multiplication:
\[K_G^i(G, \tau^n)\otimes K_G^j(G, \tau^n)\cong K_G^{i+j}(G\times G, \pi_1^*\tau^n\otimes \pi_2^*\tau^n)\stackrel{m_*}{\longrightarrow} K_G^{i+j-\text{dim }G}(G, \tau^n).\]
Note that $H_G^2(G\times G, \mathbb{Z})=0$. So there is only one equivalence class of trivializations of Morita isomorphism between $m^*\tau^n$ and $\pi_1^*\tau^n\otimes \pi_2^*\tau^n$. Hence $m_*$ is well-defined and can only be $m^0_*$ (cf. \cite[\S 4.4]{M}). The same is true of the pushforward $f_{G, *}^0: K_G^*(\text{pt})\to K_G^{*+\text{dim }G}(G, \tau^{n})$ induced by the inclusion $f_G: \text{pt}\to G$ of the identity element into $G$ as $H_G^2(\text{pt}, \mathbb{Z})=0$. The Freed-Hopkins-Teleman Theorem links the Verlinde algebra of $G$ and its equivariant twisted $K$-theory. 
\begin{theorem}\label{FHT1}(\cite[Theorem 1]{FHT3}) Let $\textsf{h}^\vee$ be the dual Coxeter number of $G$ and $k+\textsf{h}^\vee> 0$. The pushforward map
\[f_{G, *}^0: R(G)\cong K_G^*(\text{pt})\to K_G^{*+\text{dim }G}(G, \tau^{k+\textsf{h}^\vee})\]
	induced by $f_G$, the inclusion of the identity element of $G$ into $G$ is a ring epimorphism, with kernel being the level $k$ Verlinde ideal $I_k(G)$. In other words, $K_G^{\text{dim }G}(G, \tau^{k+\textsf{h}^\vee})$ is isomorphic to the level $k$ Verlinde algebra $R_k(G)\cong R(G)/I_k(G)$, and $K_G^{\text{dim }G+1}(G, \tau^{k+\textsf{h}^\vee})=0$.
\end{theorem}
Let $V_\mu\in R(G)$ be the isomorphism class of the irreducible representation of $G$ with highest weight $\mu$, $W_\mu:=f_{G, *}(V_\mu)$, $B$ the basic inner product on $\mathfrak{g}^*$, and $\Lambda_k^*$ be the set of level $k$ weights defined by 
\[\displaystyle\Lambda_k^*:=\left\{\mu\in\Lambda^*_+\left| B(\mu, \alpha_{\text{max}})\leq k,\ \alpha_{\text{max}}:=\text{the highest root}\right.\right\}.\] 
Through the link with the Verlinde algebra, $W_\mu$ can be interpreted as the isomorphism class of the irreducible positive energy representation of the loop group $LG$ of level $k$ with highest weight $\mu$, if $\mu\in\Lambda_k^*$.

There is a similar description of equivariant twisted $K$-theory of a torus $T$ with the trivial action by itself. Let $\tau$ be an equivariant twist on $T$ which is classified by an element of $H_T^2(\text{pt}, \mathbb{Z})\otimes H^1(T, \mathbb{Z})\subseteq H_T^3(T, \mathbb{Z})$, and $\chi(T)$ the character group of $T$. We analogously endow $K_T^*(T, \tau)$ with the ring structure given by the Pontryagin product induced by $m_*^0$. 
Similar to $f_{G, *}^0$, the pushforward 
\[f_{T, *}^0: K_T^*(\text{pt})\to K_T^{*+\text{dim}T}(T, \tau)\]
then is a ring homomorphism. Let $\chi(T)$ be the character group $\text{Hom}(T, U(1))$ of $T$, which is also the weight lattice $\Lambda^*$. The subgroup $H_T^2(\text{pt}, \mathbb{Z})\otimes H^1(T, \mathbb{Z})$ of the equivariant cohomology group $H_T^3(T, \mathbb{Z})$ can be identified with 
\begin{align*}
	H_T^2(\text{pt}, \mathbb{Z})\otimes H^1(T, \mathbb{Z})&\cong\chi(T)\otimes\chi(T)\\
												&\cong\text{Hom}(\pi_1(T), \chi(T))\\
												&\subseteq \text{Hom}(\mathfrak{t}, \mathfrak{t}^*).
\end{align*}
In this way $\tau$ is associated with a group homomorphism $p\in\text{Hom}(\pi_1(T), \chi(T))$. The following proposition generalizes the result on the group structure of $K_T^*(T, \tau)$ asserted in \cite[\S6.2]{FHT3}.
\begin{proposition}\label{FHTtorus} Let $T$ be a torus and $f_T: \text{pt}\to T$ be the inclusion of the identity element into $T$. Let $\tau$ be a $T$-equivariant twist on $T$ which corresponds to a full-rank homomorphism $p\in\text{Hom}(\pi_1(T), \chi(T))$. Then the pushforward map 
\[f_{T, *}^0: R(T)\cong K_T^{*}(\text{pt})\to K_T^{*+\text{dim }T}(T, \tau)\]
	is a surjective ring homomorphism with kernel being the ideal $\displaystyle\left(e^\chi-1\left|\chi\in p(\pi_1(T))\right.\right)$.
\end{proposition}
\begin{proof}
	Suppose $T$ is an $n$-dimensional torus, and for $r=1, 2, \cdots, n$, let $T_r$ be an $r$-dimensional subtorus of $T$ such that $T_r$ is a subtorus of $T_{r+1}$ (we take $T_n=T$ and $T_0=\{\text{Id}_T\}$). For $r<s$, let $f_{T_{r, s}}: T_r\to T_s$ be the inclusion map. By abuse of notation, denote the image of $e^\chi\in R(T)$ under the pushforward map $f_{T_{0, r}, *}^0: R(T)\cong K_T^*(\text{pt})\to K_T^r(T_r, \tau|_{T_r})$ by $[e^\chi]$. Note that $K_T^*(T_r, \tau|_{T_r})$ is also an $R(T)$-algebra equipped with the Pontryagin product. We shall show more generally that the $R(T)$-algebra homomorphism $f_{T_{0, r}, *}^0: R(T)\cong K_T^0(\text{pt})\to K_T^r(T_r, \tau|_{T_r})$ is onto with kernel being the ideal $\left(e^\chi-1\left|\chi\in p(\pi_1(T_r))\right.\right)$ of $R(T)$. This can be done by induction on $r$. The base case $r=0$ is obvious.
	
	We think of the circle $S^1$ as the set of complex numbers of unit modulus. Let $\mathcal{N}=S^1\setminus\{-i\}$, $\mathcal{S}=S^1\setminus\{i\}$. Then $\mathcal{N}\cap \mathcal{S}$ is the disjoint union of two open arcs $\mathcal{W}$ (on the left) and $\mathcal{E}$ (on the right). Let $t$ be the standard $S^1$-representation. So $R(S^1)\cong\mathbb{Z}[t, t^{-1}]$. Identifying $T_{r+1}$ with $T_r\times S^1$, we consider the Mayer-Vietoris sequence for $K_{T}^*(T_{r+1}, \tau|_{T_{r+1}})$ with respect to the open cover $\{T_r\times\mathcal{N}, T_r\times\mathcal{S}\}$. 
	\begin{eqnarray*}
		\xymatrix{\cdots\ar[r]& K_T^r(T_{r+1}, \tau|_{T_{r+1}})\ar[r]&K_T^r(T_r\times\mathcal{N}, \tau|_{T_r\times\mathcal{N}})\oplus K_T^r(T_r\times\mathcal{S}, \tau|_{T_r\times\mathcal{S}})\ar[r]^-{r^*}&} 
	\end{eqnarray*}
	\begin{eqnarray*}
		\xymatrix@+2pc{K_T^r(T_r\times\mathcal{W}, \tau|_{T_r\times \mathcal{W}})\oplus K_T^r(T\times\mathcal{E}, \tau|_{T_r\times\mathcal{E}})\ar[r]^-{\partial}& K_T^{r+1}(T_{r+1}, \tau|_{T_{r+1}})\ar[r]&\cdots}
	\end{eqnarray*}
	The twist $\tau|_{T_{r+1}}$ can be regarded as the restricted map $p|_{T_{r+1}}\in \text{Hom}(\pi_1(T_{r+1}), \chi(T))$. Let ${\chi_{r+1}}$ be the image of the positive generator of the fundamental group of the circle factor of $T_{r+1}$ which with $T_r$ generates $T_{r+1}$. As $p$ is of full rank, it is injective and in particular, $\chi_{r+1}$ is never the trivial representation as long as $r+1\leq n$. The restricted twist $\tau|_{T_{r+1}}$ can be thought of, with respect to the open cover, as the two twists $\tau|_{T_r\times\mathcal{N}}$ and $\tau|_{T_r\times\mathcal{S}}$ glued together with an equivariant line bundle-valued transition function on $(T_r\times\mathcal{N})\cap(T_r\times\mathcal{S})=(T_r\times\mathcal{W})\cup(T_r\times\mathcal{E})$ from $T_r\times\mathcal{N}$ to $T_r\times\mathcal{S}$. This transition function takes the trivial line bundle with a certain fiberwise $T$-representation $\chi$ on $T_r\times\mathcal{W}$ and another one with fiberwise $T$-representation $\chi-\chi_{r+1}$ on $T_r\times\mathcal{E}$ (the line bundle-valued transition function is chosen in such a way that when traversing $S^1$ counterclockwise, the total holonomy is $\chi_{r+1}$. The twists with different representations $\chi$ on $\mathcal{W}$ are all Morita isomorphic, and acted upon transitively by the equivariant Picard group of the identity $1\in T$, i.e. $H_T^2(\text{pt}, \mathbb{Z})\subset H_{T}^2(T_{r+1}, \mathbb{Z})$). The restriction $f_{T_{0, r+1}}^*\tau|_{T_{r+1}}$ at the identity $1\in T_{r+1}$ is Morita isomorphic to the trivial equivariant twist.  This isomorphism is trivialized by a Hilbert space bundle over the identity with the trivial $T$-action, because the pushforward $f^0_{T_{0, r+1},*}$ corresponds to $0\in H_T^2(\text{pt}, \mathbb{Z})$. Thus the line bundle taken by the transition function on $\mathcal{E}$ (which contains the identity $1$) has the trivial $T$-action and so $\chi=\chi_{r+1}$. 
Each of the four summands in the middle two terms of the above Mayer-Vietoris sequence is isomorphic to $K_T^r(T_r, \tau|_{T_r})$, which in turn is isomorphic to $R(T)/(e^\chi-1|\chi\in p(\pi_1(T_r)))$ by the inductive hypothesis. Note that the $R(T)$-module structure of $K_T^r(T_r, \tau|_{T_r})$ is given by $e^{\chi_1}\cdot[e^\chi_2]=i^*(e^{\chi_1})\cdot[e^{\chi_2}]$, where $i^*: K_T^*(\text{pt})\to K_T^*(T_r)$ is the map induced by the collapsing map $T_r\to \text{pt}$, and the second product is the $K$-theory product. For any $r$, $f_{T_{0, r}, *}^0$ is a $K_T^*(\text{pt})$-module homomorphism. So in $K_T^r(T_r, \tau|_{T_r})$, we have $i^*(e^{\chi_1})\cdot[e^{\chi_2}]=[e^{\chi_1+\chi_2}]$.
	
	While $r^*([1], 0)=([1], [1])$, $r^*(0, [1])$, the (signed and twisted) restriction, is $-(i^*(e^{\chi_{r+1}})\cdot[1], [1])=(-[e^{\chi_{r+1}}], -[1])$ due to the line bundle-valued transition function of the twist $\tau$. The coboundary map $\partial$ is onto as $K_T^{r+1}(T_r \times\mathcal{N}, \tau|_{T_r\times\mathcal{N}})\oplus K_T^{r+1}(T_r\times\mathcal{S}, \tau|_{T_r\times\mathcal{S}})$, the term following $K_T^{r+1}(T_{r+1}, \tau|_{T_{r+1}})$, is 0 by inductive hypothesis. Thus as a module over $K_T^r(T_r, \tau|_{T_r})$ (and over $R(T)$), $K_T^r(T_{r+1}, \tau|_{T_{r+1}})$ is isomorphic to 
	\begin{align*}
		&\left(R(T)/(e^\chi-1|\chi\in p(\pi_1(T_r)))\right)\left/\left(\det\begin{pmatrix}[1] &[1]\\ -[e^{\chi_{r+1}}]& -[1]\end{pmatrix}\right)\right.\\
		\cong &(R(T)/(e^\chi-1| \chi\in p(\pi_1(T_{r}))))/(e^{\chi_{r+1}}-1)\\
		\cong&R(T)/(e^\chi-1| \chi\in p(\pi_1(T_{r+1}))).	
	\end{align*}
	Next, we shall show that $f_{T_{0, r+1}, *}^0$ is onto. This time we treat $K_T^*(T_{r+1}, \tau|_{T_{r+1}})$ as the $K$-homology $K_{r+1-*}^T(T_{r+1}, \tau|_{T_{r+1}})$ by Poincar\'e duality and consider its Mayer-Vietoris sequence with respect to the same open cover, together with the various pushforward maps induced by the inclusion of $T_0$ into $T_{r+1}$. 
	\begin{eqnarray*}
		\xymatrix@+1pc{\cdots\ar[r]& K_0^T(\text{pt})\oplus K_0^T(\text{pt})\ar[r]\ar[d]^{f_{T_{0, r}, *}^0\oplus f_{T_{0, r}, *}^0}& K_0^T(\text{pt})\ar[d]^{f_{T_{0, r+1}, *}^0}\ar[r]&\cdots\\ \cdots\ar[r]& K_0^T(T_r\times\mathcal{N}, \tau|_{T_r\times\mathcal{N}})\oplus K_0^T(T_r\times\mathcal{S}, \tau|_{T_r\times \mathcal{S}})\ar[r]^-{i_*^{\mathcal{N}}-i_*^{\mathcal{S}}}& K_0^T(T_{r+1}, \tau|_{T_{r+1}})\ar[r]^-\partial&\cdots}
	\end{eqnarray*}
	Here, the bottom row is the Mayer-Vietoris sequence for $K_0^T(T_{r+1}, \tau|_{T_{r+1}})$ while the top row is the one for $K_0^T(\text{pt})\cong R(T)$. By the commutativity of the square in the above diagram and the surjectivity of $f_{T_{0, r}, *}^0\oplus f_{T_{0, r},*}^0$ (by inductive hypothesis) and $i_*^\mathcal{N}-i_*^\mathcal{S}$ (as $K_{-1}^T(T_r\times\mathcal{N}, \tau|_{T_r\times\mathcal{N}})\oplus K_{-1}^T(T_r\times\mathcal{S}, \tau|_{T_r\times\mathcal{S}})$, the term following $K_0^T(T_{r+1}, \tau|_{T_{r+1}})$, is 0 by inductive hypothesis), we get the desired surjectivity of $f_{T_{0, r+1}, *}^0$. Being a surjective ring homomorphism, $f_{T_{0, r+1}, *}^0$ gives the unique $R(T)$-algebra structure to $K_T^{r+1}(T_{r+1}, \tau|_{T_{r+1}})$, which can be seen now to be isomorphic to $R(T)/(e^\chi-1|\chi\in p(\pi_1(T_{r+1})))$ as an $R(T)$-algebra. The kernel of $f_{T_{0, r+1}, *}^0$ obviously is the ideal $(e^\chi-1| \chi\in p(\pi_1(T_{r+1})))$. The inductive step is complete and hence this finishes the proof.
\end{proof}
\subsection{The restriction map $i^*: K_G^*(G, \tau^{n})\to K_T^*(T, i^*\tau^{n})$}
Let $i: T\to G$ be the inclusion of a maximal torus $T$ into $G$, and $n$ an arbitrary integer. 
\begin{proposition}
	Let $j^*: K_G^*(G, \tau^n)\to K_T^*(G, \tau^n)$ be the map restricting the $G$-action to the $T$-action. Then the map
	\begin{align*}
		R(T)\otimes_{R(G)}K_G^*(G, \tau^n)&\to K_T ^*(G, \tau^n)\\
		 u\otimes v&\mapsto u\cdot j ^*(v)
	\end{align*}
	is an algebra isomorphism.
\end{proposition}	
\begin{proof}
	Let $G$ act on $G/T$ by the left action. We can apply the equivariant K\"unneth spectral sequence to $K_G^*(G/T\times G, 1\otimes\tau^n)$, which is isomorphic to $K_T^*(G, \tau^n)$. The spectral sequence collapses on the $E_2$-page, which is isomorphic to $K_G^*(G/T)\otimes_{R(G)}K_G^*(G, \tau^n)\cong R(T)\otimes_{R(G)}K_G^*(G, \tau^n)$, because $K_G^*(G/T)\cong R(T)$ is a free $R(G)$-module (with the Steinberg basis). Composing the above isomorphisms give the map in the Proposition. A priori this map is a $R(T)$-module isomorphism, but it is also compatible with the algebra structure induced by the Pontryagin product. So the map is an algebra isomorphism.
\end{proof}

\begin{proposition}\label{injective}
	The restriction map $i^*: K_G^*(G, \tau^n)\to K_T^*(T, i^*\tau^n)$ is injective. 
\end{proposition}
\begin{proof}
	Let $W$ be the Weyl group $N_G(T)/T$. By \cite[Theorem 7.8]{FHT3}, the pullback $\omega^*: K_G^*(G, \tau^n)\to K_G^*(G/T\times_W T, \omega^*\tau^n)$ induced by the Weyl map
	\begin{align*}
		\omega: G/T\times_W T&\to G\\
		[gT, t]&\mapsto gtg ^{-1}
	\end{align*}
	is injective. Note that $i^*$ is the composition of $\omega^*$ and the restriction map
	\[K_G^*(G/T\times_W T, \omega^*\tau^n)\cong K_{T\rtimes W}^*(T, i^*\tau^n)\to K_T^*(T, i^*\tau^n)\]
	which by the discussion in \cite[\S6.2]{FHT3} is also injective. Hence $i^*$ is injective.
\end{proof}
Let $B_G^\flat: \mathfrak{g}\to\mathfrak{g}^*$ be the linear transformation induced by the basic inner product on $\mathfrak{g}$. 
In this way, $i^*\tau^{n}$ can be seen as the group homomorphism $nB_{G}^\flat|_\mathfrak{t}$ (cf. \cite[Lemma 3.1]{M}). We then can identify $K_T^*(T, i^*\tau^{n})$ with $\displaystyle \frac{R(T)}{\left(e^\chi-1\left|\chi\in nB_{G}^\flat|_\mathfrak{t}(\pi_1(T))\right.\right)}$ as $R(T)$-algebras by Proposition \ref{FHTtorus}. To describe the restriction map $\displaystyle i^*: K_G^*(G, \tau^{n})\to K_T^*(T, i^*\tau^{n})$, it suffices to find out $i^*(W_0)$, as 
for any $\mu\in\Lambda_k^*$, $i^*(W_\mu)=i^*(V_\mu\cdot W_0)=i^*(V_\mu)\cdot i^*(W_0)$, where $\cdot$ means module multiplication. We denote the image of $e^\chi\in R(T)$ under the quotient map $f_{T, *}^0: R(T)\to K_T^*(T, i^*\tau^{n})$ by $[e^\chi]$.
\begin{proposition}
	Let $n$ be a nonzero integer and $\rho$ be the half sum of positive roots. Then $\displaystyle i^*(W_0)=\left[(-1)^{|R_+|}J(e^\rho)\right]$, where $R_+$ is the set of positive roots and $\displaystyle J(e^\rho):=\sum_{w\in W}(-1)^w e^{w\cdot\rho}$ is the Weyl denominator. 
\end{proposition}
\begin{proof}
	
	Let $f_T:\text{pt}\to T$ and $f_G: \text{pt}\to G$ be the inclusion of the identity element into $T$ and $G$ respectively. Consider the maps
	\[K_T^*(\text{pt})\stackrel{f_{T, *}^0}{\longrightarrow}K_T^*(T, i^*\tau^{n})\stackrel{i_*^0}{\longrightarrow}K_T^*(G, \tau^{n}).\]
	The composition $i_*^0\circ f_{T, *}^0$ is $f_{G, *}^0\otimes_{R(G)}\text{Id}_{R(T)}$ and both maps $f_{T, *}^0$ and $f_{G, *}^0\otimes_{R(G)}\text{Id}_{R(T)}$ are surjective $R(T)$-algebra homomorphisms. It follows that $i_*^0([1])=1_{R(T)\otimes_{R(G)}K_G^*(G, \tau^n)}$. Next, consider the maps
	\[K_T^*(T, i^*\tau^{n})\stackrel{i_*^c}{\longrightarrow} K_T^*(G, \tau^{n})\stackrel{i^*}{\longrightarrow} K_T^*(T, i^*\tau^{n}), \]
	where $c$ is the first Chern class of the normal bundle $N$ of $T$ in $G$. The $K$-theory class of $N$ is $\displaystyle \sum_{\alpha\in R_+} e^\alpha\otimes 1\in R(T)\otimes K^*(T)\cong K_T^*(T)$. The $K$-theory class of the line bundle with the first Chern class $c$ is $\displaystyle e^{\sum_{\alpha\in R_+}\alpha}\otimes 1$. The composition $i^*\circ i_*^c$ is the multiplication by the $K$-theoretic Euler class $\displaystyle e(N)=\prod_{\alpha\in R_+}(1-e^\alpha)\otimes 1$. Thus 
	\[i^*\circ i_*^c([1])=\left[\prod_{\alpha\in R_+}(1-e^\alpha)\right].\]
	On the other hand, 
	\begin{align*}
		i_*^c([1])&=i_*^0([1]\cdot [L_{\frac{c}{2}}])\\
				&=e^\rho\cdot 1_{R(T)\otimes_{R(G)}K_G^*(G, \tau^n)}\ (\text{the multiplication is }R(T)\text{-module multiplication}).
	\end{align*}
	It follows that 
	\begin{align*}
		i^*(1_{R(T)\otimes_{R(G)}R_k(G)})&=e^{-\rho}\cdot i^*\circ i^c_*([1])\\
									&=e^{-\rho}\cdot\left[\prod_{\alpha\in R_+}(1-e^\alpha)\right]\\
									&=\left[\prod_{\alpha\in R_+}(e^{-\frac{\alpha}{2}}-e^{\frac{\alpha}{2}})\right]\\
									&=\left[(-1)^{|R_+|}J(e^\rho)\right].	
	\end{align*}
	This, together with the fact that the map $K_G^*(G, \tau^{n})\to K_T^*(G, \tau^{n})$ sends $W_0$ to $1_{R(T)\otimes_{R(G)}K_G^*(G, \tau^n)}$, finishes the proof.
	\end{proof}
	\begin{corollary}\label{antisym}
		$\displaystyle i^*(W_\mu)=\left[(-1)^{|R_+|}J(e^{\mu+\rho})\right]$.
	\end{corollary}
	\begin{proof}
		\begin{align*}
			i^*(W_\mu)&=i^*(V_\mu\cdot W_0)\ (\text{the multiplication is }R(G)\text{-module multiplication})\\
					&=i^*(V_\mu)\cdot i^*(W_0)\ (\text{the multiplication is }R(T)\text{-module multiplication})\\
					&=i^*(V_\mu)\cdot\left[(-1)^{|R_+|}J(e^\rho)\right]\\
					&=\left[(-1)^{|R_+|}J(e^{\mu+\rho})\right]\ \text{(by the Weyl Character Formula)}
		\end{align*}
	\end{proof}

\section{Adams operations: the equivariant twisted case}\label{adamsequivcase}
In this section we assume that $G$ is a simple, connected and simply-connected compact Lie group unless otherwise specified. 
The Adams operation $\psi^\ell$ for $\ell\in\mathbb{Z}$ is a cohomological operation on $K$-theory. For a $G$-equivariant complex vector bundle $E$ over $X$, one can think of it as the direct sum of equivariant line bundles $\displaystyle\bigoplus_{i=1}^n L_i$ by the equivariant splitting principle. The Adams operation $\psi^\ell$ on $E$ for $\ell\geq 0$ is defined to be 
\[\psi^\ell(E)=\bigoplus_{i=1}^n L_i^{\otimes\ell}.\]
For $\ell<0$, $\psi^\ell(E)$ is defined to be $\psi^{-\ell}(E^*)=\bigoplus_{i=1}^\ell (L_i^{*})^{\otimes(-\ell)}$. Thus $\psi^\ell$ descends to a map on $K_G^0(G)$. By stipulating that $\psi^\ell$ commute with the suspension isomorphism, the Adams operation then is defined on the full $K$-theory ring $K_G^*(X)$. 

There is an equivalent formulation of Adams operations by means of representations of permutation groups, which can be generalized to twisted $K$-theory (\cite[\S2]{A}, \cite[\S10]{AS2}). Consider the representation group $R(S_\ell)$ of the permutation group $S_\ell$, $\ell>0$. A $\mathbb{Z}$-basis of $R(S_\ell)$, apart from the one consisting of irreducibles, is given by 
\[\{\text{Ind}_{S_{\alpha_1}\times S_{\alpha_2}\times\cdots\times S_{\alpha_r}}^{S_\ell}(\text{triv})|\alpha_1+\cdots+\alpha_r=\ell, \alpha_i\in\mathbb{N}, 1\leq i\leq r\leq \ell\}.\]
The tensor power $E^{\otimes\ell}$, which is naturally a $S_\ell$-representation through permutation of tensor factors, can be decomposed with respect to this basis into the direct sum 
\[\bigoplus_{\alpha\text{ a partition of }\ell}\text{Ind}_{S_{\alpha_1}\times S_{\alpha_2}\times\cdots\times S_{\alpha_r}}^{S_\ell}(\text{triv})\otimes E_\alpha.\]
The Adams operations $\psi^\ell(E)$ then is $E_\alpha$ where $\alpha$ is the partition $\ell$. This construction can be naturally extended to complexes of equivariant vector bundles and hence equivariant $K$-theory. For twisted $K$-theory $K_H^*(X, \tau)$, let $P$ be a projective Hilbert space bundle over $X$ whose Morita isomorphism class is $\tau$. One can similarly decompose $P^{\otimes \ell}$ with respect to the basis of induced representations and extract the subbundle $P_\ell$ corresponding to the partition $\ell$, which we denote by $\psi^\ell(P)$ and has Morita isomorphism class $\ell\tau$. The Adams operation $\psi^\ell: K_H^*(X, \tau)\to K_H^*(X, \ell\tau)$ for $\ell\geq 0$ then can be defined by the map on the representing spaces $\text{Fred}(P)\to\text{Fred}(\psi^\ell(P))$ induced by $\psi^\ell$. For $\ell<0$, define $\psi^\ell$ by requiring that $\psi^\ell(P):=\psi^{-\ell}(P^*)$.


Rather than sticking to the above definition of Adams operations, we will essentially make use of the following facts when computing the Adams operations on the equivariant twisted $K$-theory of compact Lie groups. Note that for $r\in R(H)$ and $\alpha\in K_H^*(X, \tau)$, we have that $\psi^\ell$ respects the $R(H)$-module structure of $K_H^*(X, \tau)$ in the sense that 
\begin{eqnarray}\label{adamsmodule}\psi^\ell(r\cdot\alpha)=\psi^\ell(r)\cdot\psi^\ell(\alpha).\end{eqnarray}
Thus Adams operations on $K_T^*(T, \tau)$ are completely determined by $\psi^\ell([1])$ as the pushforward map $f_{T, *}^0: R(T)\to K_T^*(T, \tau)$ is onto by Proposition \ref{FHTtorus}. 

\begin{lemma}
	The Adams operation $\psi^\ell: K_T^{\text{dim }T}(T, \tau)\to K_T^{\text{dim }T}(T, \ell\tau)$ satisfies
	\[\psi^\ell([1])=[1].\]
\end{lemma}
\begin{proof}
	As in the proof of Proposition \ref{FHTtorus}, we will show more generally that for $0\leq r\leq n$, the Adams operation $\psi^\ell: K_T^r(T_r, \tau|_{T_r})\to K_T^r(T_r, \ell\tau|_{T_r})$ satisfies $\psi^\ell([1])=[1]$, using induction. In the base case $r=0$, $[1]\in K_T^0(\text{pt})$ is just the trivial representation and so the lemma is obviously true in this case. For the inductive step, consider the boundary map $\partial$ of the Mayer-Vietoris sequence for $K_T^{r+1}(T_{r+1}, \tau|_{T_{r+1}})$ in the proof of Proposition \ref{FHTtorus}. We have
	\[\partial([e^\chi], [e^{\chi'}])=\det\begin{pmatrix}[1]& [1]\\ [e^\chi]& [e^{\chi'}]\end{pmatrix}=[e^{\chi'}-e^\chi]\]
	(or equivalently $\displaystyle \det\begin{pmatrix}[e^\chi]& [e^{\chi'}]\\ -[e^{\chi_{r+1}}]& -[1]\end{pmatrix}=[e^{\chi'}\cdot e^{\chi_{r+1}}-e^\chi]=[e^{\chi'}-e^\chi]$, the last equation due to $[e^{\chi_{r+1}}-1]=0$ in $K_T^{r+1}(T_r, \tau|_{T_{r+1}})$). In particular, $\partial(0, 1)=[1]$. In fact, $\partial$ can be understood as the `twisted' suspension isomorphism between 
	\[K_{T}^{r+1}(T_{r+1}, \tau|_{T_{r+1}})\text{ and } \left(K_{T}^r(T_r\times\mathcal{W}, \tau|_{T_r\times\mathcal{W}})\oplus K_{T}^r(T_r\times\mathcal{E}, \tau|_{T_r\times\mathcal{E}})\right)\left/\text{Im}(r^*)\right.\] 
	(when $\tau$ is trivial, $\partial$ is exactly the suspension isomorphism between $K_{T}^r(T_r)$ and $K_{T}^{r+1}(T_{r+1})$). By definition Adams operations and the suspension isomorphism commute, so 
	\begin{align*}
		\psi^\ell([1])&=\psi^\ell(\partial(0, [1]))\\
				&=\partial(\psi^\ell(0, [1]))\\
				&=\partial(0, \psi^\ell([1]))\\
				&=\partial(0, [1])\ (\text{by the inductive hypothesis})\\
				&=[1].
	\end{align*}
\end{proof}
\begin{remark}
	Though Adams operations commute with suspension isomorphism, they do not commute with Bott periodicity. More precisely, if $\mu\in K^{-2}_H(\text{pt})$ is the Bott generator and $x\in K_H^*(X, \tau)$, then $\psi^\ell(\mu\cdot x)=\ell\mu\cdot\psi^\ell(x)$.
\end{remark}
\begin{corollary}\label{adamsT}
	More generally, the Adams operation $\psi^\ell: K_T^{\text{dim }T}(T, \tau)\to K_T^{\text{dim }T}(T, \ell\tau)$ satisfies
	\[\psi^\ell([e^\chi])=[e^{\ell\chi}].\]
\end{corollary}
\begin{proof}
	By Equation (\ref{adamsmodule}), we have $\displaystyle\psi^\ell([e^\chi])=\psi^\ell(e^\chi\cdot [1])=\psi^\ell(e^\chi)\cdot\psi^\ell([1])=e^{\ell\chi}\cdot[1]=[e^{\ell\chi}]$.
\end{proof}
\begin{proof}[Proof of Theorem \ref{adamsequiv}] 

For $|n|\geq\textsf{h}^\vee$ and $\ell>0$, we have
	\begin{align*}
		i^*\psi^\ell(W_\mu)&=\psi^\ell(i^*(W_\mu))\\
						&=\psi^\ell\left(\left[(-1)^{|R_+|}J(e^{\mu+\rho})\right]\right)\ (\text{by Corollary \ref{antisym}})\\
						&=\left[(-1)^{|R_+|}J(e^{\ell(\mu+\rho)})\right]\ (\text{by Corollary \ref{adamsT}})\\
						&=i^*(W_{\ell(\mu+\rho)-\rho}).
	\end{align*}
	Comparing both ends of the string of equalities and noting that $i^*$ is injective, we get the desired conclusion. For $|n|\geq \textsf{h}^\vee$ and $\ell=-1$, note that, on the one hand, 
\begin{align*}
	&\psi^{-1}(i^*(W_\mu))\\
	=&\psi^{-1}(\left[(-1)^{|R_+|}J(e^{\mu+\rho})\right])\ (\text{by Corollary \ref{antisym}})\\
	=&\left[(-1) ^{|R_+|}J(e^{-\mu-\rho})\right]\ (\text{By Corollary \ref{adamsT}})\\
	=&\left[(-1) ^{|R_+|+\text{sgn}(w_0)}J(e^{w_0(-\mu-\rho)})\right]\\
	=&\left[(-1) ^{|R_+|+\text{sgn}(w_0)}J(e^{\mu^*+\rho})\right].
\end{align*}
On the other hand, by the functoriality of $\psi^\ell$, $\psi^{-1}(i^*(W_\mu))=i^*(\psi^{-1}(W_\mu))$. By Corollary \ref{antisym} again, we have $\psi^{-1}(W_\mu)=(-1)^{\text{sgn}(w_0)}W_{\mu^*}$. The case $|n|\geq \textsf{h}^\vee$ and $\ell<0$ follows by applying the formulas for $\psi^{-\ell}$ and $\psi^{-1}$ we have just got to the equation $\psi^\ell=\psi^{-1}\circ\psi^{-\ell}$. Observing that $J(e^0)=0$, we have that the Adams operation $\psi^0$ is the zero map. Alternatively, note that $K_G^*(G, \tau^n)$ is a torsion $R(G)$-module and $K_G^*(G)$ is a free $R(G)$-module (cf. \cite{BZ}), and any moodule homomorphism from a torsion module to a free module must be the zero map. Finally, The case $|n|<\textsf{h}^\vee$ and $n\neq 0$ follows from the fact that the Verlinde algebra $R_{k}(G)=0$ for $1-\textsf{h}^\vee\leq k\leq -1$ (i.e. $0<n<\textsf{h}^\vee$) and Corollary \ref{zerok} which says that $K_G^*(G, \tau^n)\cong K_G^*(G, \tau^{-n})$.
\end{proof}
\begin{remark}
	The formula for Adams operations on $K_G^*(G, \tau^{n})$ is in stark contrast to that on the untwisted $K$-theory, equivariant or non-equivariant, in terms of complexity (cf. \cite[Theorem 1.2, Proposition 2.6]{F}). While the formula for the twisted equivariant $K$-theory amounts to a simple scaling of the highest weight with a $\rho$-correction, the formula for the untwisted $K$-theory involves complicated combinatorial coefficients of canonical generators. 
\end{remark}
\begin{remark}
	It is interesting to note that, in untwisted $K$-theory and ordinary representation rings, the Adams operation $\psi^0$ is the augmentation map, giving the rank of vector bundles or the dimension of representations, whereas $\psi^0$ on the twisted $K$-theory of $G$ is the zero map.
\end{remark}
As noted in the Introduction, the Adams operation $\psi^{-1}$, which is induced by complex conjugation, is an involutive isomorphism because $\psi^{-1}\circ\psi^{-1}=\psi^1$, which is the identity map. Thus through $\psi^{-1}$ $K_G^*(G, \tau^n)$ is isomorphic to $K_G^*(G, \tau^{-n})$ (as abelian groups). We shall further deduce that $K_G^*(G, \tau^n)$ for $n<0$, as a ring with the Pontryagin product, is also isomorphic to a Verlinde algebra in the following
\begin{corollary}\label{zerok}
	Let $n<0$. Then the pushforward map $f_{G, *}: R(G)\cong K_G^*(\text{pt})\to K_G^{*+\text{dim }G}(G, \tau^n)$ induced by the inclusion of the identity element of $G$ into $G$ is a ring epimorphism, with kernel being the level $|n|-\textsf{h}^\vee$ Verlinde ideal $I_{|n|-\textsf{h}^\vee}$. In other words, $K_G^{\text{dim }G}(G, \tau^n)$ is isomorphic to the level $|n|-\textsf{h}^\vee$ Verlinde algebra $R_{|n|-\textsf{h}^\vee}(G)$ and $K_G^{\text{dim }G+1}(G, \tau^n)=0$.
\end{corollary}
\begin{proof}
	To avoid confusion we use $f_{G, *}^\pm$ to denote the pushforward map $K_G^*(\text{pt})\to K_G^{*+\text{dim }G}(G, \tau^{\mp n})$ (the superscript indicates the sign of the twist of the codomain). Note that, for $V_\mu\in K_G^*(\text{pt})$, 
	\begin{align*}
		f_{G, *}^+(\psi^{-1}(V_\mu))&=f_{G, *}^+(V_{\mu^*})\\
							&=W_{\mu^*}\in K_G^*(G, \tau^{-n})\\
		\psi^{-1}(f_{G, *}^-(V_\mu))&=\psi^{-1}(W_\mu)\\
							&=(-1)^{\text{sgn}(w_0)}W_{\mu^*}\in K_G^*(G, \tau^{-n})\ (\text{by Theorem \ref{adamsequiv}}).
	\end{align*}
	It follows that $\psi^{-1}\circ f_{G, *}^-=(-1)^{\text{sgn}(w_0)}f_{G, *}^+\circ \psi^{-1}$ on $K_G^*(\text{pt})$. As $\psi^{-1}$ is an involutive isomorphism and $f_{G, *}^+$ is onto by Theorem \ref{FHT1}, $f_{G, *}^-$ is onto as well. The kernel of $f_{G, *}^-$ is $\psi^-(\text{ker}(f_{G, *}^+))=\psi^{-1}(I_{|n|-\textsf{h}^\vee})$, which is just $I_{|n|-\textsf{h}^\vee}$ as the Verlinde ideal is invariant under complex conjugation (cf. \cite[Lemma 5.13]{F2}). This finishes the proof of the corollary.
\end{proof}

\section{Nonequivariant twisted $K$-theory}\label{nonequivtwist}
	The nonequivariant twisted $K$-theory $K^*(G, \tau^n)$ has been studied via various different approaches, namely, via the Freed-Hopkins-Teleman Theorem and K\"unneth spectral sequence (cf. \cite{Bra}), twisted $\text{Spin}^c$ bordism and Rothenberg-Steenrod spectral sequence (cf. \cite{D}), and twisted Segal spectral sequence applied to some classical Lie groups which can be realized as successive fiber bundles (\cite{MR}). It turns out that 
	\begin{theorem}[\cite{Bra, D}]\label{nonequiv}
		Let $r=\text{rank }G$ and $|n|\geq \textsf{h}^\vee$. The nonequivariant twisted $K$-theory $K^*(G, \tau^n)$ is isomorphic, as a ring equipped with the Pontryagin product, to the exterior algebra $\displaystyle\bigwedge\nolimits^*_{\mathbb{Z}/c(G, n)}(\eta_1, \cdots, \eta_{r-1})$. Here $c(G, n)$ is the order of the cokernel of the augmentation map $I_{|n|-\textsf{h}^\vee}\to\mathbb{Z}$, which can be found in \cite[Equation (30) and Table 1]{Bra} or \cite[Theorem 1.2]{D}. The isomorphism takes the even degree elements in the exterior algebra to $K^{\text{dim }G}(G, \tau^n)$ and the odd degree elements to $K^{\text{dim }G+1}(G, \tau^n)$. 
	\end{theorem}
	Let us briefly recall the proof of Theorem \ref{nonequiv} in \cite{Bra} and suitably generalize the argument for the case where $n<\textsf{h}^\vee$. By rewriting $K^*(G, \tau^n)$ as the equivariant $K$-theory $K_G^*(G_\text{Ad}\times G_\text{L}, \pi_1^*\tau^n\otimes 1)$ (where $G_\text{Ad}$ is $G$ equipped with the action of conjugation by itself and $G_\text{L}$ is $G$ equipped with the action of left multiplication by itself), we can apply the K\"unneth spectral sequence, which is shown in \cite{Bra} to collapse on the $E_2$-page $\text{Tor}_{R(G)}^*(K_G^*(G_\text{Ad}, \tau^n), K_G^*(G_\text{L}))=\text{Tor}_{R(G)}^*(R_{|n|-\textsf{h}^\vee}(G), \mathbb{Z})$ and have no extension problem. By a commutative algebra argument (see \cite[\S 3.4.2]{Bra}), there exist $y_1^{(n)}, \cdots, y_r^{(n)}\in I_{|n|-\textsf{h}^\vee}$ which generate $I_{|n|-\textsf{h}^\vee}$ and form a regular sequence in $R(G)$. Thus $R_k(G)$ admits a free $R(G)$-resolution by the Koszul complex $\mathcal{K}(y_1^{(n)}, \cdots, y_r^{(n)})$, which can be realized as the tensor product of complexes $\displaystyle\bigotimes_{i=1}^r\left(0\longrightarrow R(G)\stackrel{\cdot y_i^{(n)}}{\longrightarrow}R(G)\longrightarrow0\right)$. By the definition of Tor, we have 
	\begin{align*}
		\text{Tor}_{R(G)}^*(R_k(G), \mathbb{Z})&=H_*\left(\bigotimes_{i=1}^r\left(0\longrightarrow R(G)\stackrel{\cdot y_i^{(n)}}{\longrightarrow}R(G)\longrightarrow0\right)\otimes_{R(G)}\mathbb{Z}\right)\\
										&=H_*\left(\bigotimes_{i=1}^r \left(0\longrightarrow \mathbb{Z}\stackrel{\cdot \text{dim }y_i^{(n)}}{\longrightarrow}\mathbb{Z}\longrightarrow 0\right)\right).
	\end{align*}
	It can be shown by induction on $j$ that $\displaystyle H_*\left(\bigotimes_{i=1}^j \left(0\longrightarrow \mathbb{Z}\stackrel{\cdot \text{dim }y_i^{(n)}}{\longrightarrow}\mathbb{Z}\longrightarrow 0\right)\right)\cong \bigwedge\nolimits^*_{\mathbb{Z}/x_j}(\eta_1, \cdots, \eta_{j-1})$, where $x_j:=\text{gcd}(\text{dim }y_1^{(n)}, \cdots, \text{dim }y_j^{(n)})$. The base case $j=1$ is easy. For the inductive step, we apply the universal coefficient theorem to $\displaystyle H_*\left(\bigotimes_{i=1}^{j+1}\left(0\longrightarrow \mathbb{Z}\stackrel{\cdot\text{dim }y_i^{(n)}}{\longrightarrow}\mathbb{Z}\longrightarrow 0\right)\right)$ and get 
	\begin{align*}
		&H_*\left(\bigotimes_{i=1}^j\left(0\longrightarrow \mathbb{Z}\stackrel{\cdot\text{dim }y_i^{(n)}}{\longrightarrow}\mathbb{Z}\longrightarrow 0\right)\right)\otimes H_*\left(0\longrightarrow\mathbb{Z}\stackrel{\cdot\text{dim }y_{j+1}^{(n)}}{\longrightarrow}\mathbb{Z}\longrightarrow 0\right)\oplus\\
		&\text{Tor}_\mathbb{Z} ^1\left(H_*\left(\bigotimes_{i=1}^j\left(0\longrightarrow\mathbb{Z}\stackrel{\cdot\text{dim }y_i^{(n)}}{\longrightarrow}\mathbb{Z}\longrightarrow 0\right)\right), H_*\left(0\longrightarrow\mathbb{Z}\stackrel{\cdot\text{dim }y_{j+1}^{(n)}}{\longrightarrow}\mathbb{Z}\longrightarrow 0\right)\right)\\
		\cong&\bigwedge\nolimits_{\mathbb{Z}/x_j}^*(\eta_1, \cdots, \eta_{j-1})\otimes \text{coker}\left(\mathbb{Z}\stackrel{\cdot\text{dim }y_{j+1}^{(n)}}{\longrightarrow}\mathbb{Z}\right)\oplus\\
		&\text{ker}\left(\bigwedge\nolimits_\mathbb{Z}^*(\eta_1, \cdots, \eta_{j-1})\otimes\mathbb{Z}/\text{dim }y_{j+1}^{(n)}\stackrel{\cdot x_j\otimes\text{Id}_{\mathbb{Z}/{\text{dim }y_{j+1}^{(n)}}}}{\longrightarrow}\bigwedge\nolimits_\mathbb{Z}^*(\eta_1, \cdots, \eta_{j-1})\otimes\mathbb{Z}/{\text{dim }y_{j+1}^{(n)}}\right)\\
		\cong&\bigwedge\nolimits_{\mathbb{Z}/{x_j}}^*(\eta_1, \cdots, \eta_{j-1})\otimes\mathbb{Z}/{\text{dim }y_{j+1}^{(n)}}\oplus\bigwedge\nolimits^*_\mathbb{Z}(\eta_1, \cdots, \eta_{j-1})\otimes\left\langle\frac{\text{dim }y_{j+1}^{(n)}}{\text{gcd}(\text{dim }y_{j+1}^{(n)}, x_j)}\right\rangle\eta_j\\
		&(\text{Note that }\cdot x_j\otimes\text{Id}_{\mathbb{Z}/{\text{dim }y_{j+1}^{(n)}}}=\text{Id}\otimes\cdot x_j|_{\mathbb{Z}/{\text{dim }y_{j+1}^{(n)}}},\ \text{and the generator of }\\
		&\text{ker}(\mathbb{Z}/{\text{dim }y_{j+1}^{(n)}}\stackrel{\cdot x_j}{\longrightarrow}\mathbb{Z}/{\text{dim }y_{j+1}^{(n)}})\text{ is}\frac{\text{dim }y_{j+1}^{(n)}}{\text{gcd}(y_{j+1}^{(n)}, x_j)})\\
		\cong&\bigwedge\nolimits^*_{\mathbb{Z}/{x_{j+1}}}(\eta_1, \cdots, \eta_{j-1})\oplus\bigwedge\nolimits_{\mathbb{Z}}(\eta_1, \cdots, \eta_{j-1})\otimes\mathbb{Z}/{x_{j+1}}\cdot\eta_j\\
		\cong&\bigwedge\nolimits^*_{\mathbb{Z}/{x_{j+1}}}(\eta_1, \cdots, \eta_j).
	\end{align*}
	The above computation shows that in fact $c(G, n)=\text{gcd}(\text{dim }y_1^{(n)}, \cdots, \text{dim }y_r^{(n)})$, which is the positive number generating of the image of the augmentation homomorphism $I_{|n|-\textsf{h}^\vee}\to\mathbb{Z}$ and thus the order of the cokernel of the augmentation map. Moreover, each new generator $\eta_j$ comes from the first copy of $\mathbb{Z}/{\text{dim }y_{j+1}^{(n)}}$ in the sequence $\displaystyle 0\longrightarrow\mathbb{Z}/{\text{dim }y_{j+1}^{(n)}}\stackrel{\cdot x_j}{\longrightarrow}\mathbb{Z}/{\text{dim }y_{j+1}^{(n)}}\longrightarrow 0$ as the generator of the kernel of the multiplication by $x_j$ map and its degree is $\text{dim }G+1$ (the second copy of $\mathbb{Z}/{\text{dim }y_{j+1}^{(n)}}$ originates from $K^{\text{dim }G}(G, \tau^n)$ and thus its degree is $\text{dim }G$). The coefficient group $\mathbb{Z}/c(G, n)$ of $\displaystyle\bigwedge\nolimits^*_{\mathbb{Z}/c(G, n)}(\eta_1, \cdots, \eta_{r-1})$ (that is, the degree zero part of the exterior algebra) corresponds to $\displaystyle\bigotimes_{i=1}^r\text{coker}\left(\mathbb{Z}\stackrel{\cdot\text{dim }y_i^{(n)}}{\longrightarrow}\mathbb{Z}\right)$ in $\text{Tor}_{R(G)}^*(R_{|n|-\textsf{h}^\vee}(G), \mathbb{Z})$. We shall remark that the set of algebra generators $\{\eta_1, \cdots, \eta_{r-1}\}$ depends on the choice of the generators $y_1^{(n)}, \cdots, y_r^{(n)}$ of $I_{|n|-\textsf{h}^\vee}$ which form a regular sequence. 
	\begin{proposition}\label{forgetful}
		Let $|n|\geq \textsf{h}^\vee$. The image of the forgetful map $\displaystyle h: K_G^*(G, \tau^n)\to K^*(G, \tau^n)\cong\bigwedge\nolimits^*_{\mathbb{Z}/c(G, n)}(\eta_1, \cdots, \eta_{r-1})$ is the coefficient group $\mathbb{Z}/c(G, n)$, and
		\[h(W_\mu)=\text{dim }V_\mu\ (\text{mod }c(G, n)).\]
	\end{proposition}
	\begin{proof}
		The equivariant twisted $K$-theory $K^*_G(G, \tau^n)$ can be realized as $\text{Tor}_{R(G)}(R_{|n|-\textsf{h}^\vee}, R(G))$, the $E_2$-page of the K\"unneth spectral sequence applied to $K_G^*(G_\text{Ad}\times\text{pt}, \tau^n\otimes 1)$. Using the Koszul complex and the universal coefficient theorem as in the previous discussion, we can see that $\displaystyle \text{Tor}_{R(G)}^*(R_{|n|-\textsf{h}^\vee}(G), R(G))=\bigotimes_{i=1}^r\text{coker}\left(R(G)\stackrel{\cdot y_i^{(n)}}{\longrightarrow}R(G)\right)$, which after applying the forgetful functor (amounting to taking $\cdot\otimes_{R(G)}\mathbb{Z}$ or equivalently the augmentation map) becomes $\displaystyle\bigotimes_{i=1}^r\text{coker}\left(\mathbb{Z}\stackrel{\cdot\text{dim }y_i^{(n)}}{\longrightarrow}\mathbb{Z}\right)$. The latter corresponds to the coefficient group $\mathbb{Z}/c(G, n)$ as asserted in the previous discussion. This completes the proof.
	\end{proof}
	\begin{remark}
		Recall that the untwisted equivariant $K$-theory $K_G^*(G)$ is isomorphic to $\Omega^*_{R(G)/\mathbb{Z}}$, the Grothendieck ring of K\"ahler differentials of $R(G)$ (cf. \cite{BZ}), while the classical result of Hodgkin asserts that the untwisted, nonequivariant $K$-theory $K^*(G)$ is the exterior algebra over $\mathbb{Z}$ generated by differentials of $R(G)$ (cf. \cite{Ho}). As the $G$-action on itself by conjugation is equivariantly formal, the forgetful map $K_G^*(G)\to K^*(G)$ is onto (cf. \cite{F3}). However, because of the presence of twists, the forgetful map $h$ for twisted $K$-theory is far from being onto by Proposition \ref{forgetful}.
	\end{remark}
	It remains to deal with the case where $|n|<\textsf{h}^\vee$ and $n\neq 0$, but this is easy: as mentioned in the proof of Corollary \ref{zerok}, $K_G^*(G, \tau^n)=0$ for $|n|<\textsf{h}^\vee$ and $n\neq 0$, and thus the $E_2$-page of the K\"unneth spectral sequence is 
	\begin{align*}
		\text{Tor}_{R(G)}^*(K^*_G(G, \tau^n), \mathbb{Z})&=\text{Tor}_{R(G)}^*(0, \mathbb{Z})\\
											&=0.
	\end{align*}
	So we have
	\begin{proposition}\label{noneqvanishing}
		If $|n|<\textsf{h}^\vee$ and $n\neq 0$, then $K^*(G, \tau^n)=0$.
	\end{proposition}
\section{Adams operations: the nonequivariant twisted case}\label{nonequivadams}
The Adams operation on $K^*(G, \tau^n)$ for $|n|<\textsf{h}^\vee$ is automatically trivial by Proposition \ref{noneqvanishing}. In this section, unless otherwise specified, we assume that $|n|\geq \textsf{h}^\vee$.
\begin{theorem}[=Theorem \ref{adamsnonequiv}]
	Let $k$ be an element in the coefficient group $\mathbb{Z}/c(G, n)$ of $K^*(G, \tau^n)$. Then 
	\[\psi^\ell(k)=\begin{cases}\ell^{|R_+|}k,&\ \text{if }\ell\geq0\\ (-1)^{\text{sgn}(w_0)+|R_+|}\ell^{|R_+|}k,&\ \text{if }\ell<0\end{cases}\ (\text{mod }c(G, \ell n))\]
	in the coefficient group.
\end{theorem}
\begin{proof}
	First, let us recall the Weyl dimension formula, which computes the dimension of the irreducible representation $V_\mu$ of $G$: 
	\[\text{dim }V_\mu=\prod_{\alpha\in R_+}\frac{\langle\mu+\rho, \alpha\rangle}{\langle\rho, \alpha\rangle}.\]
	Here $\langle\cdot, \cdot\rangle$ is any $G$-invariant inner product on $\mathfrak{g}^*$, e.g. the basic inner product $B_G$. Note that the Adams operation $\psi^\ell$ commutes with $h$ due to functoriality. Let 1 be the multiplicative identity of $K^*(G, \tau^n)$. We have, for $\ell>0$, 
	\begin{align*}
		\psi^\ell(1)&=\psi ^\ell(h(W_0))\ (\text{by Proposition \ref{forgetful}})\\
				&=h(\psi^\ell(W_0))\\
				&=h(W_{(\ell-1)\rho})\ (\text{by Theorem \ref{adamsequiv}})\\
				&=\prod_{\alpha\in R_+}\frac{\langle\ell\rho, \alpha\rangle}{\langle\rho, \alpha\rangle}\ (\text{mod }c(G, \ell n))\ (\text{by the Weyl dimension formula and Proposition \ref{forgetful}})\\
				&=\ell ^{|R_+|}\ (\text{mod }c(G, \ell n)).
	\end{align*}
	For the case $\ell<0$, apply Theorem \ref{zerok} to the equation $\psi^{\ell}=\psi^{-1}\circ\psi^{-\ell}$ and use the previous case. For the case $\ell=0$, $\psi^0$ is the zero map since its equivariant counterpart is by Theorem \ref{adamsequiv}. The proposition follows by the fact that $\psi^\ell$ is a homomorphism.
\end{proof}
Finding out how Adams operations act on the (wedge product of) algebra generators $\eta_1, \cdots, \eta_{r-1}$ takes more work. This amounts to understanding the map $\text{Tor}_{R(G)}(R_{|n|-\textsf{h}^\vee}(G), \mathbb{Z})\to \text{Tor}_{R(G)}(R_{|\ell n|-\textsf{h}^\vee}(G, \mathbb{Z}))$ induced by the equivariant Adams operation $\psi^\ell: R_{|n|-\textsf{h}^\vee}(G)\to R_{|\ell n|-\textsf{h}^\vee}(G)$ given by Theorem \ref{adamsequiv}. We shall first consider the Koszul complexes for $K_G^*(G, \tau^n)\cong R_{|n|-\textsf{h}^\vee}(G)$ and $K_G^*(G, \tau^{\ell n})\cong R_{|\ell n|-\textsf{h}^\vee}(G)$ and figure out the map between them induced by the Adams operation $\psi^\ell$. 
\begin{eqnarray*}
	\xymatrix@+3pc{\cdots\ar[r]&\bigwedge\nolimits^2_{R(G)} (R(G)^{\oplus r})\ar[r]\ar[d]^{\widetilde{\psi}_2^\ell}& R(G)^{\oplus r}\ar[r]^{(y_1^{(n)}, \cdots, y_r^{(n)})}\ar[d]^{\widetilde{\psi}^\ell_1}& R(G)\ar[r]\ar[d]^{\widetilde{\psi}^\ell_0}& 0\\ \cdots\ar[r]&\bigwedge\nolimits^2_{R(G)} (R(G)^{\oplus r})\ar[r]& R(G)^{\oplus r}\ar[r]^{(y_1^{(\ell n)}, \cdots, y_r^{(\ell n)})}& R(G)\ar[r]&0}
\end{eqnarray*}
By Theorem \ref{adamsequiv}, $\psi^\ell$ is induced by the map $\widetilde{\psi}_1^\ell: R(G)\to R(G)$ in the zeroth slot of the complexes which sends $V_\mu$ to $V_{\ell\mu+(\ell-1)\rho}$. In general the map $\widetilde{\psi}^\ell_i: \bigwedge\nolimits_{R(G)}^i (R(G)^{\oplus r})\to \bigwedge\nolimits^i_{R(G)}(R(G)^{\oplus r})$ induced by $\psi^\ell$ in the $i$-th slot of the complexes can be found successively by commutativity of the diagram. For example, $\widetilde{\psi}_1^\ell(0, 0, \cdots, 0, \underbrace{1}_{i\text{-th entry}}, 0, \cdots, 0)$ is $(m_1^{(n, i)}, \cdots, m_r^{(n, i)})\in R(G)^{\oplus r}$ which satisfies 
\[\sum_{j=1}^r m_j^{(n, i)}y_j^{(n\ell)}=\widetilde{\psi}^\ell_0(y_i^{(n)}).\]
After obtaining the maps $\widetilde{\psi}_i^\ell$, we apply the functor $\cdot\otimes_{R(G)}\mathbb{Z}$ and from there the action of $\psi^\ell$ on the algebra generators can be computed. Let us illustrate this with an example. 
\begin{example}
	Let $G=SU(3)$. Without loss of generality we assume that $n\geq\textsf{h}^\vee=3$ because by Corollary \ref{zerok} $\psi^\ell$ is the zero map in the case $-3<n<3$ and the case $n\leq-3$ is similar. Take $y_1^{(n)}=V_{(n-2)L_1}$, $y_2^{(n)}=V_{(n-1)L_1}$ as the two generators of the Verlinde ideal $I_{n-3}$ (cf. \cite{G}). Assume that $\ell\geq 2$. Recall the Giambelli's formula, which expresses any irreducible representation of $SU(n)$ in terms of the fundamental representations: 
	\[V_{a_1L_1+a_2L_2+\cdots+a_{n-1}L_{n-1}}=\det(V_{(a_i+j-i)L_1})_{1\leq i, j\leq n-1}.\]
	By Giambelli's formula, We have 
	\begin{align*}
		V_{(\ell-1)L_1}\cdot y_1^{(\ell n)}-V_{(\ell-2)L_1}\cdot y_2^{(\ell n)}&=\widetilde{\psi}_0^\ell(y_1^{(n)})=V_{(\ell n-2)L_1+(\ell-1)L_2},\\
		-V_{(\ell-2)L_1+(\ell-2)L_2}\cdot y_1^{(\ell n)}+V_{(\ell-1)L_1+(\ell-1)L_2}\cdot y_2 ^{(\ell n)}&=\widetilde{\psi}_0 ^\ell(y_2^{(n)})=V_{(\ell n+\ell-2)L_1+(\ell-1)L_2}.
	\end{align*}
	On the other hand, after applying $\cdot\otimes_{R(G)}\mathbb{Z}$ to the above diagram, we have that $\displaystyle\text{dim }y_1^{(n)}=\frac{n(n-1)}{2},\ \text{dim }y_2^{(n)}=\frac{(n+1)n}{2},\ \text{dim }y_1^{(\ell n)}=\frac{\ell n(\ell n-1)}{2},\ \text{dim }y_2^{(\ell n)}=\frac{\ell n(\ell n+1)}{2}$, and the algebra generators $\eta_1^{(n)}\in K^*(G, \tau^n)$ and $\eta_1^{(\ell n)}\in K^*(G, \tau^{\ell n})$ are respectively represented by the generators of the kernel of the maps
	\[\mathbb{Z}^{\oplus 2}\stackrel{\left(\frac{n(n-1)}{2}, \frac{n(n+1)}{2}\right)}{\longrightarrow}\mathbb{Z}\text{ and }\mathbb{Z}^{\oplus 2}\stackrel{\left(\frac{\ell n(\ell n-1)}{2}, \frac{\ell n(\ell n+1)}{2}\right)}{\longrightarrow}\mathbb{Z},\]
	i.e., 
	\[\eta_1^{(n)}=\begin{cases}[(n+1, 1-n)],&\ \text{if }n\text{ is even}\\ \left[\left(\frac{n+1}{2}, \frac{1-n}{2}\right)\right],&\ \text{if }n\text{ is odd}\end{cases}, \text{ and}\] 
	\[\displaystyle\eta_1^{(\ell n)}=\begin{cases}[(\ell n+1, 1-\ell n)],&\ \text{if }\ell\text{ or }n\text{ is even}\\ \left[\left(\frac{\ell n+1}{2}, \frac{1-\ell n}{2}\right)\right],&\ \text{if }\ell\text{ and }n\text{ are odd}\end{cases}.\]
	Moreover, we have
	\begin{align*}
		&(\widetilde{\psi}_1^\ell\otimes_{R(G)}\text{Id}_\mathbb{Z})(n+1, 1-n)\\
		=&(n+1)(\text{dim }V_{(\ell-1)L_1}, -\text{dim }V_{(\ell-2)L_1})+(1-n)(-\text{dim }V_{(\ell-2)L_1+(\ell-2)L_2}, \text{dim }V_{(\ell-1)L_1+(\ell-1)L_2})\\
		=&(n+1)\left(\frac{(\ell+1)\ell}{2}, -\frac{\ell(\ell-1)}{2}\right)+(1-n)\left(-\frac{\ell(\ell-1)}{2}, \frac{\ell(\ell+1)}{2}\right)\\
		=&(\ell(\ell n+1), -\ell(\ell n-1)).
	\end{align*}
	Thus on the $K$-theoretic level, we obtain, for $\ell\geq 2$, 
	\[\psi^\ell (\eta_1^{(n)})=\begin{cases}\frac{\ell}{2}\eta_1^{(\ell n)},&\ \text{if }n\text{ is odd and }\ell\text{ is even},\\ \ell\eta_1^{(\ell n)},&\ \text{otherwise}.\end{cases}\]
	One can reprise the same recipe for the case $\ell<0$. Note that $y_1^{(\ell n)}=V_{(-\ell n-2)L_1}$, $y_2^{(\ell n)}=V_{(-\ell n-1)L_1}$, and $\displaystyle\eta_1^{(\ell n)}=\begin{cases}[(-\ell n+1, 1+\ell n)],&\ \text{if }\ell\text{ or }n\text{ is even}\\ \left[\left(\frac{-\ell n+1}{2}, \frac{1+\ell n}{2}\right)\right],&\ \text{if }\ell\text{ and }n\text{ are odd}\end{cases}.$ We also have, by Giambelli's formula and Theorem \ref{adamsequiv},
	\begin{align*}
		-V_{(-\ell(n-1)-1)L_1}\cdot y_1^{(\ell n)}+V_{(-\ell(n-1)-2)L_1}\cdot y_2 ^{(\ell n)}&=\widetilde{\psi}_0 ^\ell(y_1^{(n)})=-V_{(-\ell n-2)L_1+(-\ell(n-1)-1)L_2},\\
		V_{(-\ell(n+1)-1)L_1}\cdot y_1^{(\ell n)}-V_{(-\ell(n+1)-2)L_1}\cdot y_2 ^{(\ell n)}&=\widetilde{\psi}_0 ^\ell(y_2^{(n)})=-V_{(-\ell(n+1)-2)L_1+(\ell n-1)L_2}.
	\end{align*}
	It follows that
	\begin{align*}
		&(\widetilde{\psi}_1^\ell\otimes_{R(G)}\text{Id}_\mathbb{Z})(n+1, 1-n)\\
		=&(n+1)(-\text{dim }V_{(-\ell(n-1)-1)L_1}, \text{dim }V_{(-\ell(n-1)-2)L_1})+\\
		&(1-n)(\text{dim }V_{(-\ell(n+1)-1)L_1}, -\text{dim }V_{(-\ell(n+1)-2)L_1})\\
		=&(n+1)\left(-\frac{(-\ell(n-1)+1)(-\ell(n-1))}{2}, \frac{-\ell(n-1)(-\ell(n-1)-1)}{2}\right)+\\
		&(1-n)\left(\frac{(-\ell(n+1)+1)(-\ell(n+1))}{2}, -\frac{(-\ell(n+1))(-\ell(n+1)-1)}{2}\right)\\
		=&((1-\ell n)\ell(n^2-1), (1+\ell n)\ell(n^2-1)).
	\end{align*}
	So for $\ell<0$, we get
	\[\psi^\ell (\eta_1^{(n)})=\begin{cases}\frac{\ell(n^2-1)}{2}\eta_1^{(\ell n)},&\ \text{if }n\text{ is odd and }\ell\text{ is even},\\ \ell(n^2-1)\eta_1^{(\ell n)},&\ \text{otherwise}.\end{cases}\]
	Recall that, by \cite[Equation 3.20 and Table 1]{Bra}, 
	\[c(SU(3), \ell n)=\begin{cases}\ell n,&\ \text{if }\ell n\text{ is odd},\\ \frac{\ell n}{2},&\ \text{if }\ell n\text{ is even}. \end{cases}\]
	Then the above results for $\ell<0$ can be simplified as
		\[\psi^\ell (\eta_1^{(n)})=\begin{cases}-\frac{\ell}{2}\eta_1^{(\ell n)},&\ \text{if }n\text{ is odd and }\ell\text{ is even},\\ -\ell\eta_1^{(\ell n)},&\ \text{otherwise}.\end{cases}\]
\end{example}

\noindent\footnotesize{\textsc{NYU-ECNU Institute of Mathematical Sciences, \\
New York University Shanghai, \\
3663 Zhongshan Road North,\\
Shanghai 200062, China\\
\\
Xi'an Jiaotong-Liverpool University,\\
111 Ren’ai Road, Suzhou Industrial
Park,\\Suzhou, Jiangsu Province 215123, China}\\
\\
\textsc{E-mail}: \texttt{ChiKwong.Fok@xjtlu.edu.cn}\\
\textsc{URL}: \texttt{https://sites.google.com/site/alexckfok}

\begin{thebibliography}{9999}

\bibitem[A]{A} M. F. Atiyah, 
	\emph{Power operations in K-theory}, 
	Quart. J. Math. (2), 17, pp. 165--93, 1966.
	
\bibitem[AH]{AH} M. F. Atiyah, F. Hirzebruch, 
	 \emph{Vector bundles and homogeneous spaces}, 
	 AMS Symposium in Pure Math. III, 1960.	
	 
\bibitem[AS]{AS}
 M. Atiyah and G. Segal, 
 \emph{Twisted K-theory}, 
 Ukr. Mat. Visn. 1 (2004), no. 3, 287--330.
 	
\bibitem[AS2]{AS2}
M. F. Atiyah, G. Segal, 
\emph{Twisted K-theory and cohomology}, 
Nankai Tracts in Mathematics: Inspired by S S Chern, pp. 5-43, 2006.


\bibitem[Bou]{Bou}
A. K. Bousfield, \emph{The K-theory localizations and $\nu_1$-periodic homotopy groups of H-spaces,}
Topology, Vol. 38, No. 6, pp. 1239--1264, 1999.

\bibitem[Bra]{Bra}
V. Braun, \emph{Twisted K-theory of Lie groups}, 
J. of High Energy Phys. 03 (2004) 029.

\bibitem[BZ]{BZ}
J.-L. Brylinski, B. Zhang, 
	\emph{Equivariant K-theory of compact connected Lie groups}, 
	K-theory, Vol. 20, pp. 23-36, 2000.

\bibitem[CW]{CW} 
A. Carey, B-L. Wang, 
\emph{Thom isomorphism and push-forward map in twisted K-theory}, 
Journal of K-theory: K-theory and its Applications to Algebra, Geometry, and Topology, 1, 357-393, 2008.

\bibitem[D]{D}
C. Douglas, 
\emph{On the twisted K-homology of simple Lie groups}, 
Topology 45, 955-988, 2006

\bibitem[F]{F} 
C.-K. Fok, \emph{Adams operations on classical compact Lie groups}, 
Proc. Amer. Math. Soc., Vol. 145, pp. 2799-2813, 2017.

\bibitem[F2]{F2}
C.-K. Fok, \emph{Equivariant twisted Real K-theory of compact Lie groups}, 
Journal of Geometry and Physics, Vol. 124, pp. 325--349, 2018.

\bibitem[F3]{F3}
C.-K. Fok, \emph{Equivariant formality in K-theory}, 
New York J. Math. Vol. 25, pp. 315--327, 2019.

\bibitem[FHT3]{FHT3}
D. S. Freed, M. J. Hopkins, and C. Teleman, 
\emph{Loop groups and twisted K-theory III},
Ann. of Math. (2), 174(2), pp. 947--1007, 2011.

\bibitem[G]{G}
D. Gepner, \emph{Fusion rings and geometry}, 
Commun. Math. Phys. Vol. 141, pp. 381--411, 1991.

\bibitem[Ho]{Ho} 
	L. Hodgkin, 
	\emph{On the K-theory of Lie groups}, 
	Topology, Vol. 6, pp. 1-36, 1967.

\bibitem[MR]{MR} 
V. Mathai and J. Rosenberg, 
\emph{Group dualities, T-dualities, and twisted K-theory}, 
J. Lond. Math. Soc. (2), Vol. 97, No. 1, pp. 1--23, 2018.

\bibitem[M]{M} 
	E. Meinrenken, 
	\emph{On the quantization of conjugacy classes}, 
	L'Enseignement Mathematique, Vol. 55, 33-75, 2009.
\end{thebibliography}
\end{document}